\documentclass[a4paper,12pt]{article}
\usepackage[latin1]{inputenc} 
\usepackage[T1]{fontenc}
\usepackage[english]{babel}
\usepackage{a4wide}

\usepackage[dvipsnames]{xcolor} 
\usepackage{framed} 
\colorlet{shadecolor}{red!20}
\usepackage{indentfirst} 
\usepackage{latexsym}
\usepackage{amssymb}
\usepackage{amsmath}
\usepackage{mathtools}
\usepackage{amsthm}
\usepackage{url}
\usepackage[stable]{footmisc}
\usepackage[backend=biber,style=numeric]{biblatex}
\usepackage[hidelinks]{hyperref}
\usepackage{enumitem}
\setlist[description]{leftmargin=\parindent,labelindent=\parindent}

\newtheorem{thm}{Theorem}[section]
\newtheorem{cor}[thm]{Corollary}
\newtheorem{lem}[thm]{Lemma}
\newtheorem{prop}[thm]{Proposition}
\newtheorem{definition}[thm]{Definition}

\newtheorem{exam}{Example}[section]

\newtheorem*{thm*}{Theorem}

\newcommand{\U}{\mathcal{U}}
\newcommand{\N}{\mathbb{N}}

\newcommand{\Z}{\mathbb{Z}}
\newcommand{\Q}{\mathbb{Q}}
\newcommand{\C}{\mathbb{C}}

\newcommand{\V}{\mathcal{V}}

\newcommand{\bN}{\beta\mathbb{N}}

\newcommand{\supp}{\operatorname{supp}}
\renewcommand{\phi}{\varphi}
\newcommand{\prodint}[2]{ #1\cdot #2}
\newcommand{\mat}[3]{\operatorname{Mat}_{#1\times #2}(#3)}

\newcommand{\vecspan}{\operatorname{span}}

\renewcommand{\hat}{\widehat}
\renewcommand{\check}{\widecheck}
\makeatletter
\DeclareRobustCommand\widecheck[1]{{\mathpalette\@widecheck{#1}}}
\def\@widecheck#1#2{%
    \setbox\z@\hbox{\m@th$#1#2$}%
    \setbox\tw@\hbox{\m@th$#1%
       \widehat{%
          \vrule\@width\z@\@height\ht\z@
          \vrule\@height\z@\@width\wd\z@}$}%
    \dp\tw@-\ht\z@
    \@tempdima\ht\z@ \advance\@tempdima2\ht\tw@ \divide\@tempdima\thr@@
    \setbox\tw@\hbox{%
       \raise\@tempdima\hbox{\scalebox{1}[-1]{\lower\@tempdima\box
\tw@}}}%
    {\ooalign{\box\tw@ \cr \box\z@}}}
\makeatother

\makeatletter
\newcommand{\subjclass}[2][1991]{%
  \let\@oldtitle\@title%
  \gdef\@title{\@oldtitle\footnotetext{#1 \emph{Mathematics Subject Classification:} #2}}%
}
\newcommand{\keywords}[1]{%
  \let\@@oldtitle\@title%
  \gdef\@title{\@@oldtitle\footnotetext{\emph{Keywords:} #1.}}%
}
\makeatother

\numberwithin{equation}{section}

\title{Rado functionals and applications}
\date{\today}

\author{

    Paulo Henrique Arruda \thanks{Universit\"{a}t Wien, Fakult\"{a}t f\"{u}r Mathematik, Oskar-Morgenstern-Platz 1, 1090 Vienna, Austria. Corresponding author. \url{paulo.arruda@univie.ac.at}. Supported by grant P30821-N35 of the Austrian Science Fund FWF.}
    
    \and

    Lorenzo Luperi Baglini \thanks{Dipartimento di Matematica, Universit\`{a} di Milano, Via Saldini 50, 20133 Milano, Italy. \url{lorenzo.luperi@unimi.it}. Supported by grant P30821-N35 of the Austrian Science Fund FWF.}

}

\date{\today}

\subjclass[2020]{Primary 05D10, 11B75; Secondary 11U10, 54D80.}
\keywords{partition regularity of equations; ultrafilters; Rado conditions}

\begin{document}
\maketitle
\abstract{
We study Rado functionals and the maximal condition (first introduced in \cite{BarretMoreiraLupiniMoreira2021}) in terms of the partition regularity of mixed systems of linear equations and inequalities. By strengthening the maximal Rado condition, we provide a sufficient condition for the partition regularity of polynomial equations over some infinite subsets of a given commutative ring. By applying these results, we derive an extension of a previous result obtained in \cite{LuperiBagliniDiNasso2018} concerning partition regular inhomogeneous polynomials in three variables and also conditions for the partition regularity of equations of the form $H(xz^\rho,y)=0$, where $\rho$ is a non-zero rational and $H\in\Z[x,y]$ is a homogeneous polynomial. 
}
\section{Introduction}

Through this article, $R$ will denote a commutative ring with unity; given any $S\subseteq R$, we denote by $S^{\times}$ the set $S\setminus\{0\}$. As usual, we denote the set of all polynomials over $R$ by $R[x_1,\dots,x_n]$ and, given $m,n\in\N$, we let $\mat{m}{n}{R}$ be the set of all $m\times n$ matrices with entries in $R$. 

A long studied problem in combinatorics is the partition regularity of systems of Diophantine equations.

\begin{definition}\label{definition:PR}
Given $P_1,\dots,P_m\in R[x_1,\dots,x_n]$, we say that the system of equations
\begin{equation*}
    \sigma\left(x_{1},\dots,x_{n}\right):=\left\{
        \begin{matrix}
            P_1\left(x_1,\dots,x_n\right) & = & 0\\
            \vdots & \vdots & \vdots\\
            P_m\left(x_1,\dots,x_n\right) & = & 0\\
        \end{matrix}
    \right.
\end{equation*}
is partition regular over $S$ if for every finite coloring\footnote{A finite coloring of $S$ is a function $c:S\to \{1,\dots,r\}$ for some $r\in\N$.} $c$ of $S$ there are $c$-monochromatic\footnote{Namely all belonging to $c^{-1}(i)$ for some $i\leq r$.} $a_1,\dots,a_n\in S$ satisfying $\sigma\left(a_1,\dots,a_n\right)=0$. 

We say that $\sigma\left(x_{1},\dots,x_{n}\right)=0$ is infinitely partition regular over $S$ if for each coloring $c$ of $S$ there are infinitely many $c$-monochromatic $n$-tuples $\left(a_1,\dots,a_n\right)\in S^n$ satisfying $\sigma\left(a_1,\dots,a_n\right)=0$ for each $i\leq m$. 

When $m=1$, we will simply say that the polynomial $P_{1}$ is (infinitely) partition regular to mean that the equation $P_{1}\left(x_1,\dots,x_n\right)=0$ is (infinitely) partition regular.
\end{definition}

In 1933, R. Rado completely characterized which systems of linear equations are partition regular over $\N$ in terms of the so-called columns condition, that we formulate here in a more general version for $R$.

\begin{definition}
Let $A\in\mat{m}{n}{R}$ and $C_1,\dots,C_m$ be the columns of $A$; we say that $A$ satisfies the columns condition if there is a partition $I_0,\dots,I_r$ of $[m]$ such that
\begin{enumerate}
    \item $\sum_{i\in I_0}C_i=\vec{0}$; and 
    \item given any $u\in [r]$, $\sum_{i\in I_u}C_i \in \vecspan_K\{C_j:j\in I_0\cup\dots\cup I_{u-1}\}$, where $K$ is the field of fractions of $R$.
\end{enumerate}
\end{definition}

\begin{thm}\cite{Rado1933}\label{theorem:RadosTheorem} (Rado's Theorem)
Given a matrix $A\in\mat{m}{n}{\Q}$, the system $A\vec{t}=\vec{0}$ is partition regular over $\N$ if and only if $A$ satisfies the columns condition.
\end{thm}

In \cite{Rado1945}, Rado proved the analogous of Theorem \ref{theorem:RadosTheorem} for subrings of $\C$. Also, in the same article, Rado proved the characterization of all possible inhomogeneous linear systems that are partition regular, which reads as follows:

\begin{thm}\label{theorem:PR_linear_inhomogeneous}
Given $A\in\mat{m}{n}{\Z}$ and $\vec{b}\in \Z^m$, the system $A\vec{t}=\vec{b}$ is partition regular over $\N$ if and only if either 
\begin{enumerate}
    \item there is a constant solution $s\in\N$, i.e. $A(s,\dots,s)=\vec{b}$; or
    \item $A$ satisfies the columns condition and there is a constant solution $s\in \Z$.
\end{enumerate}
\end{thm}

We call \emph{Rado system} any system of linear homogeneous equations that is partition regular. Recently, the following generalization of Rado's Theorem was proved: 

\begin{thm} \cite[Theorem A]{ByszewskiKrawczyk2021}\label{theorem:Rado_polish} Let $D$ be an infinite integral domain and $A\in\mat{m}{n}{D}$. Then, the system $Ax=0$ is partition regular over $D$ if and only if $A$ satisfies the columns condition.
\end{thm}

Several generalizations of Rado's Theorem were proved for commutative rings and infinite integral domains \cite{BergelsonDeuberHindmanLefmann1994,ByszewskiKrawczyk2021}. Although the literature for (finite and infinite\footnote{We will not discuss infinite systems here; we just mention that, whilst the finite case is settled, the analogous problem of which infinite systems of linear equations are partition regular is still open.}) linear systems is quite extensive and general, the nonlinear case is scarce and mostly restricted to $\Z$. The articles \cite{LuperiBagliniArruda2022,BarretMoreiraLupiniMoreira2021,ChowLindqvistPrendiville2021,FarhangiMagner2021,Moreira2017,LuperiBagliniDiNasso2018,Prendiville2021} contain the latest results regarding the partition regularity of nonlinear equations that we are aware of. 

In this article we will build up upon results first proven by J.M. Barret, M. Lupini and J. Moreira in \cite{BarretMoreiraLupiniMoreira2021} as generalizations of preliminary results proven by M. Di Nasso and the first author of this present paper in \cite{LuperiBagliniDiNasso2018}, to study Rado sets, Rado functionals and their implications for the partition regularity of equations. The case of polynomials in three variables will then be studied more in detail.

This paper is structured as follows: In Section \ref{section:Rado_functionals} we recall the basic definitions of Rado partitions, sets and functionals for a given polynomial $P\in R[x_1,\dots,x_n]$, and we prove some implications and equivalences of these concepts in terms of partition regularity of systems of linear equations and inequalities. In Section \ref{section:rado_conditions} we revise the maximal Rado condition and provide an straightening of the Rado condition that is sufficient for the partition regularity over rings. Section \ref{section:examples_3_var} is devoted to apply Rado functionals to the case of inhomogeneous polynomials in three variables, in particular we give a necessary and sufficient condition for the partition regularity of equations $H(xz^\rho,y)=0$, where $\rho\in\Q^\times$ and $H\in\Z[x,y]$ is an homogeneous polynomial.

\section{Rado Functionals}\label{section:Rado_functionals}

\subsection{Lower and Upper Rado Functionals}\label{subsection:upper_Rado_functionals}

Building on some nonstandard characterizations first introduced in \cite{LuperiBagliniDiNasso2018}, in \cite{BarretMoreiraLupiniMoreira2021} J. M. Barret et al. introduced the notions of Rado sets and (upper and lower) Rado functionals and proved a necessary condition for the partition regularity of Diophantine equations, namely the maximal Rado condition (that we will discuss in detail in Section \ref{section:rado_conditions}). In this section, we provide some explicit characterizations of Rado sets and functionals in term of the partition regularity of mixed systems of linear equalities and inequalities. This will allow to deduce necessary conditions on the structure of Rado sets. 

We start by recalling the basic definitions from \cite{BarretMoreiraLupiniMoreira2021}.

\begin{definition}\label{definition:monochromatic_positive_map}
Let $\phi:\Z^n\to \Z$ be a linear map with coefficients $t_1,\dots,t_n\in \N$, i.e. $\phi(a_1,\dots,a_n)=t_1a_1+\dots + t_n a_n$. Given a coloring $c$ of $\N$, we say that $\phi$ is $c$-monochromatic if its coefficients are $c$-monochromatic, i.e. if $c(t_1)=\dots=c(t_n)$.
\end{definition}

A multi-index is any element $\alpha=(\alpha(1),\dots,\alpha(n))\in\N^{n}_0$ for some $n\in\mathbb{N}$. Let $\N_{0}^{<\omega}=\bigcup_{n\in\N}\N_{0}^{n}$. Given $\alpha\in\N_0^{<\omega}$, we let $\boldsymbol{x}^\alpha:=x_1^{\alpha(1)}\cdot\dots \cdot x_{n}^{\alpha(n)}$, so that $|\alpha|=\alpha(1)+\dots+\alpha(n)$ is the \emph{degree} of $\boldsymbol{x}^{\alpha}$. Given any polynomial $P\in R[x_1,\dots,x_n]$, for every $\alpha\in\N_0^{<\omega}$ there is a $c_\alpha\in R$ such that $P(\boldsymbol{x})=\sum_{\alpha} c_\alpha\boldsymbol{x}^\alpha$ and the set $\supp(P):=\{\alpha\in\N_0^{n}:c_\alpha\neq 0\}$ (called the \emph{support} of $P$) is finite. We say that $J\subseteq\supp(P)$ is \emph{homogeneous} if for all $\alpha,\beta\in J$, $|\alpha|=|\beta|$; in particular, $P$ is homogeneous if and only if $\supp(P)$ is.

\begin{definition}
Given a polynomial $P\in R[x_1,\dots,x_n]$, a coloring $c$ for $\N$ and a $c$-monochromatic linear map $\phi$, let $M_0,\dots,M_l$ be an enumeration of $\phi[\supp(P)]$; we say that a partition $J_0,\dots,J_l$ of $\supp(P)$ is determined by $\phi$ if for every $i\in[0,l]$, $J_i = \phi^{-1}[\{M_i\}]$, i.e. $J_i$ is the fiber $\{\alpha\in\supp(P):\phi(\alpha)=M_i\}$. 

A Rado partition of the support of $P$ is a tuple $(J_0,\dots,J_l)$ such that for all colorings $c$ of $\N$, there exist infinitely many $c$-monochromatic linear maps $\phi$ such that $J_0,\dots,J_l$ is a partition for $\supp(P)$ determined by $\phi$. 

A Rado set over $P$ is any $J\subseteq \supp(P)$ such that there are a Rado partition $(J_0,\dots,J_l)$ and $i\in[0,l]$ such that $J=J_i$. 
\end{definition}

When trying to find all possible Rado partitions of the support of $P$, the first question to answer is which subsets of $\supp(P)$ can be Rado sets. In one direction, a trivial characterization can be given in terms of partition regular systems.

\begin{lem}\label{lemma:M_J_1}
Let $J=\{\alpha_1,\dots,\alpha_k\}$ be a Rado set for a polynomial $P\in R[x_1,\dots,x_n]$; then, given any $j\in [k]$, the matrix 
\begin{equation*}
    M_{j}(J) = 
    \begin{pmatrix}
        \alpha_1(1) - \alpha_j(1) & \alpha_1(2)-\alpha_j(2) & \dots & \alpha_1(n) - \alpha_j(n) \\
        \alpha_2(1) - \alpha_j(1) & \alpha_2(2)-\alpha_j(2) & \dots & \alpha_2(n) - \alpha_j(n) \\
                 \vdots          &          \vdots         & \dots & \vdots   \\      
        \alpha_{j-1}(1) - \alpha_{j}(1) & \alpha_{j-1}(2) - \alpha_{j}(2) &\dots & \alpha_{j-1}(n) - \alpha_{j}(n) \\
        \alpha_{j+1}(1) - \alpha_{j}(1) & \alpha_{j+1}(2) - \alpha_{j}(2) &\dots & \alpha_{j+1}(n) - \alpha_{j}(n) \\
                \vdots          &          \vdots         & \dots & \vdots   \\   
        \alpha_k(1) - \alpha_{j}(1) & \alpha_k(2) - \alpha_{j}(2) &\dots & \alpha_k(n) - \alpha_{j}(n) \\        
    \end{pmatrix}
\end{equation*}
satisfies the columns condition.
\end{lem}

\begin{proof}
Giving a colouring $c$ for $\N$, there must be (infinitely many) $c$-monochromatic positive linear maps $\phi:\Z^n\to \Z$ such that, for each $u,v\in[k]$, $\phi(\alpha_u)=\phi(\alpha_v)$; if $t_1,\dots,t_n\in\N$ are the coefficients of $\phi$, we have that
\begin{equation*}
    \big( \alpha_u(1)-\alpha_v(1)\big)t_1 + \dots + \big( \alpha_u(n)-\alpha_v(n)\big)t_n = 0.
\end{equation*}
Hence, picking any $j\in[k]$, we have that $M_{j}(J_i)(t_1,\dots,t_n)^{\operatorname{T}}=\vec{0}$ and, by definition, $t_1,\dots,t_n$ are $c$-monochromatic. Thus, $M_{j}(J_i)$ satisfies the columns condition.
\end{proof}

Let us observe that the columns condition is a property that is preserved under Gaussian operations; as such if $J=\{\alpha_1,\dots,\alpha_k\}$ is a Rado set for $P\in\Z[x_1,\dots,x_n]$ and $M_1(J)$ satisfies the columns condition, then for each $j\in[2,k]$, the matrix $M_j(J)$ also satisfies the columns condition. Hence, it is enough to work with the matrix $M(J) := M_1(J)$.

Lemma \ref{lemma:M_J_1} cannot be reversed, in general, as being a Rado set is a condition that involves the whole support of $P$: for example, if $P(x,y,z)=x^{4}+y^{4}z^{2}+x^{2}y^{2}z$, the set $\{(4,0,0),(0,4,2)\}$ satisfies the conclusion of Lemma \ref{lemma:M_J_1}, but it is not a Rado set as any linear map $\varphi:\N_{0}^{3}\rightarrow\N_{0}$ with $\varphi((4,0,0))=\varphi((0,4,2))=c$ necessarily gives also $\varphi((2,2,1))=c$.

However, the following result shows that Lemma \ref{lemma:M_J_1} can be reversed if we add a maximality hypothesis on $J$.

\begin{prop}\label{lemma:M_J_2}
Let $J\subseteq\supp(P)$ be such that $M(J)$ satisfies the columns condition but, for any $\alpha\in\supp(P)\setminus J$, $M(J\cup\{\alpha\})$ does not satisfy the columns condition. Then, $J$ is a Rado set.
\end{prop}
\begin{proof}
Let $c$ be a given finite coloring of $\N_0$. Define 
\begin{equation*}
    T(c) = \{\vec{t}\in\ker M(J):\vec{t}\text{ is }c\text{-monochromatic}\}.
\end{equation*}
Let $K=\supp(P)\setminus J$ and let $\mathfrak{P}$ be the collection of all possible partitions of $K$. Given $\vec{t}\in T(c)$, let $M_1,\dots,M_l$ be the possible values of the map $\alpha\mapsto \prodint{\alpha}{\vec{t}}$ ($\alpha\in\Z^n$) applied to $K$; define $\mathcal{P}_{\vec{t}}$ as the partition of $J$ determined by this map, i.e. $\mathcal{P}_{\vec{t}} = \{K_1,\dots,K_l\}$, where for each $i\in[l]$, $K_i=\{\beta\in K:\beta\cdot\vec{t}=M_i\}$. Since $T(c)$ is infinite and $\mathfrak{P}$ is finite, the set
\begin{equation*}
    S(c) = \big\{\mathcal{P}\in\mathfrak{P}: \{\vec{t}\in T(c): \mathcal{P}=\mathcal{P}_{\vec{t}}\}\text{ is infinite}\big\}
\end{equation*}
is not empty.

We claim that 
\begin{equation*}
\mathcal{S} = \{S(c):c\text{ is a finite coloring of }\N_0\}
\end{equation*}
has the finite intersection property. Indeed, for each $v\in\N$ and $i\in[v]$, let $c_i:\N_0\to\{1,\dots,m_i\}$ be finite coloring of $\N_0$. Let $p_1<\dots<p_v$ be prime numbers and $m=\prod_{i=1}^{v}p_i^{m_i}$; define $c:\N_0\to\{1,\dots, m\}$ as 
\begin{equation*}
    c(a) = \prod_{i=1}^{v} p_i^{c_i(a)}.
\end{equation*}

Let us observe that $T(c) = T(c_1)\cap\dots\cap T(c_v)$; this implies that $S(c)\subseteq S(c_1)\cap\dots\cap S(c_v)$ and settles the claim. As $\mathcal{S}$ is finite, there must be a $\mathcal{P}\in\bigcap\mathcal{S}$.

Finally, if $c$ is any finite coloring of $\N_0$, there must be infinitely monochromatic $\vec{t}$ such that $\mathcal{P}=\mathcal{P}_{\vec{t}}$. By construction, for any distinct $K_1,K_2\in\mathcal{P}$, $\alpha\in K_1$ and $\beta\in K_2$ we must have $\phi_{\vec{t}}(\alpha)\neq\phi_{\vec{t}}(\beta)$. Moreover, if $\alpha\in K$ and $\beta\in J$ are such that $\phi_{\vec{t}}(\alpha)=\phi_{\vec{t}}(\beta)$, then one can prove that $M(J\cup\{\alpha\})$ satisfies the columns condition, which contradicts the hypothesis. Consequently, $\{J\}\cup\mathcal{P}$ is a Rado Partition for $\supp(P)$, which makes $J$ a Rado set.
\end{proof}

Rado partitions are used to introduce upper and lower Rado functionals for a polynomial. 

\begin{definition}\label{definition:rado_functionals}
Let $P\in R[x_1,\dots,x_n]$ and $m\in\N_0$. Given an integer $l\geq m$ and $d_1,\dots,d_m\in\N$, a tuple $(J_0,\dots,J_l,d_1,\dots,d_m)$ is said to be a lower Rado functional of order $m$ for $P$ if for all $r\in\N$ and for all colorings $c$ of $\N_0$ there are infinitely many $c$-monochromatic positive liner maps $\phi$ such that 
\begin{enumerate}
    \item $(J_0,\dots,J_l)$ is the partition determined by $\phi$; and 
    \item if $M_0<\dots<M_l$ is the enumeration of $\supp(P)$ such that $J_i = \phi^{-1}[M_i]$, then
    \begin{enumerate}
        \item for each $i\in[m]$, $M_i-M_0=d_i$; and 
        \item $M_{m+1}-M_m\geq r$.
    \end{enumerate}
\end{enumerate}

Given $d_0,\dots,d_{m-1}\in\N$, we say that $(J_0,\dots,J_l,d_0,\dots,d_{m-1})$ is an upper Rado functional of order $m$ for $P$ if for all $r\in\N$ and for all colorings $c$ of $\N_0$ there are infinitely many $c$-monochromatic positive liner maps $\phi$ such that 
\begin{enumerate}
    \item $(J_0,\dots,J_l)$ is the partition determined by $\phi$; and 
    \item if $M_l<\dots<M_0$ is the enumeration of $\supp(P)$ such that $J_i = \phi^{-1}[M_i]$, then
    \begin{enumerate}
        \item for each $i\in[0,m-1]$, $M_0-M_i=d_i$; and 
        \item $M_{m}-M_{m+1}\geq r$.
    \end{enumerate}
\end{enumerate}
\end{definition}

Let us fix a notation that will help to deal with Rado functionals: given an upper Rado functional $\mathcal{J}=(J_0,\dots,J_l,d_0,\dots,d_{m-1})$ for $P\in R[x_1,\dots,x_n]$ let, for each $i\in\{0,\dots,l\}$, $J_i=\{\alpha_{i,1},\dots,\alpha_{i,k_i}\}$. We let
\begin{equation}\label{equation:def_matrix_urf}
    \hat{N}_\mathcal{J}=
    \begin{pmatrix}
        \alpha_{0,1}-\alpha_{m,1}\\
        \vdots\\
        \alpha_{m-1,1}-\alpha_{m,1}
    \end{pmatrix}
    \;\text{ and }\;
    \hat{A}_\mathcal{J} = 
    \begin{pmatrix}
        M(J_0)\\
        \vdots\\
        M(J_l)\\
        \hat{N}
    \end{pmatrix},
\end{equation}
where the matrices $M(J_{i})$ have been defined in Lemma \ref{lemma:M_J_1}. By the definition above, $\hat{A}_\mathcal{J}$ has $u=k_0+\dots+k_l-l+m-1$ lines. Define $\hat{b}\in\Z^u$ as 
\begin{equation}\label{equation:def_vec_urf}
    \hat{b} = 
    \begin{cases}
        0, &\text{ if } i\leq k_0+\dots+k_l-l-1; \\
        d_i, &\text{ if } k_0+\dots+k_l-l \leq i \leq k_0+\dots+k_l-l+m-1.
    \end{cases}
\end{equation}
Similarly, if $\mathcal{J}=(J_0,\dots,J_l,d_0,\dots,d_{m-1})$ is a lower Rado functional, we define the associated matrices 
\begin{equation*}
    \check{N}_\mathcal{J}=
    \begin{pmatrix}
        \alpha_{m,1}-\alpha_{0,1}\\
        \vdots\\
        \alpha_{m,1}-\alpha_{0,1}
    \end{pmatrix}
    \;\text{ and }\;
    \check{A}_\mathcal{J} = 
    \begin{pmatrix}
        M(J_0)\\
        \vdots\\
        M(J_l)\\
        \check{N}
    \end{pmatrix},
\end{equation*}
and  $\check{b}\in\Z^u$ as 
\begin{equation*}
    \check{b} = 
    \begin{cases}
        0, &\text{ if } i\leq k_0+\dots+k_l-l-1; \\
        d_i, &\text{ if } k_0+\dots+k_l-l \leq i \leq k_0+\dots+k_l-l+m-1.
    \end{cases}\tag{$\sharp\sharp$}
\end{equation*}

Lemma \ref{lemma:M_J_1} can be easily generalized to a result that characterizes upper Rado functionals in terms mixed systems of linear equalities and inequalities. To this end, we need to recall a definition.

\begin{definition}\label{definition:pr_with_delimiters}
Let $b\in\Z^n$. We say that $b\cdot \vec{t} \gg 0$ is partition regular over $\N$ iff for all coloring $c$ of $\N$ and all $l\in\N$ there exists a $c$-monochromatic $\vec{t}\in\N^n$ such that $b\cdot\vec{t}>l$.
\end{definition}

Upper Rado functionals (and lower Rado functionals) can be easily characterized as follows:

\begin{thm}\label{theorem:PR_URF_system}
Let $P\in\Z\left[x_{1},\dots,x_{m}\right]$. Then $\mathcal{J}=(J_0,\dots,J_l,d_0,\dots,d_{m-1})$ is an upper Rado functional for $P\in R[x_1,\dots,x_n]$ if and only if the system
\begin{equation}\label{eq:inhom}
    \left\{
    \begin{matrix}
        \hat{A}_\mathcal{J}\vec{t} = \hat{b}\\
        (\alpha_{m,1}-\alpha_{m+1,1})\cdot\vec{t} & \gg& 0\\
        \vdots & \vdots & \vdots\\
        (\alpha_{m,1}-\alpha_{l,1})\cdot\vec{t} & \gg& 0\\
    \end{matrix}
    \right.
\end{equation}
is infinitely partition regular over $\N$. Similarly, $\mathcal{J}=(J_0,\dots,J_l,d_1,\dots,d_{m})$ is a lower Rado functional if and only if the system 
\begin{equation}\label{eq:inhom}
    \left\{
    \begin{matrix}
        \check{A}_\mathcal{J}\vec{t} = \check{b}\\
        (\alpha_{m+1,1}-\alpha_{m,1})\cdot\vec{t} & \gg& 0\\
        \vdots & \vdots & \vdots\\
        (\alpha_{l,1}-\alpha_{m,1})\cdot\vec{t} & \gg& 0\\
    \end{matrix}
    \right.
\end{equation}
is infinitely partition regular over $\N$.
\end{thm}
\begin{proof} The proof is just a modification of that of Lemma \ref{lemma:M_J_1}. 
\end{proof}
Given any $A\in\mat{m}{n}{\Q}$, the system $Ax=0$ is partition regular over $\N$ if and only if it is infinitely partition regular; nevertheless, the same does not apply for inhomogeneous linear systems, as the equation $x+y+z=3$ is partition regular over $\N$ but only admits one monochromatic solution, namely $x=y=z=1$. The following result is a simple consequence of Theorem \ref{theorem:PR_linear_inhomogeneous} and characterize infinitely partition regular inhomogeneous linear systems; in later sections, we apply the following result to the study of Rado functionals.

\begin{lem}\label{lemma:infinitely_pr_inhomogeneous_linear}
Given $A\in\mat{m}{n}{\Z}$ and $\vec{b}\in \Z^m\setminus\{0\}$, suppose that the system $A\vec{t}=\vec{b}$ is partition regular and $A$ does not satisfy the columns condition. Then the system $A\vec{t}=\vec{b}$ is not infinitely partition regular. Conversely, if $A$ satisfies the columns condition and there is a constant solution $s\in\Z$ to the system $A\vec{t}=\vec{b}$, then the system $A\vec{t}=\vec{b}$ is infinitely partition regular. 
\end{lem}

\begin{proof}
Suppose that $A\vec{t}=\vec{b}$ is partition regular and $A$ does not satisfy the columns condition. Let $k\in\N$, $c:\N\to[k]$ be a coloring; by Condition 2 of Theorem \ref{theorem:PR_linear_inhomogeneous}, there is a $s\in\N$ that is a constant solution to $A\vec{t}=\vec{b}$. Define the coloring $\chi:\N\to[k+1]$ as 
\begin{equation*}
    \chi(x) = 
    \begin{cases}
        c(x-s),& \text{ if } x>s;\text{ or }\\
        k+1,&\text{ otherwise.}
    \end{cases}
\end{equation*}
Since $A\vec{t}=\vec{b}$ is partition regular, one can find a $a_1,\dots,a_n\in\N$ $\chi$-monochromatic such that $A(a_1,\dots,a_n)=\vec{b}$. But this implies that $A(a_1-s,\dots,a_n-s)=0$, and thus we must have that there are $i,j\in[n]$ such that $c(a_i-s)\neq c(a_j-s)$, which implies that $\chi(a_1)=\dots=\chi(a_n)=k+1$ and, by the definition of $\chi$, that $a_1,\dots,a_n\leq s$. 

Conversely, if $A$ satisfies the columns condition and there is $s\in\Z$ that is a constant solution of the system $A\vec{t}=\vec{b}$, then given any $k\in\N$ and any coloring $c:\N\to[k]$, the coloring $\chi$ as defined above produces infinitely many $\chi$-monochromatic $a_1,\dots,a_n\in\N$ such that $A(a_1,\dots,a_n)=0$ and $a_1,\dots,a_n>|s|$. Then, $a_1-s,\dots,a_n-s$ are $c$-monochromatic solutions to $A\vec{t}=\vec{b}$.
\end{proof}

Theorem \ref{theorem:PR_URF_system} has a few interesting consequences, that force strict conditions on Rado partitions. In all the results below, we keep the same notations introduced above. We omit the consequences for the lower Rado functionals, as they are stated and derived in a totally similar fashion. Through the rest of this section $P\in R\left[x_{1},\dots,x_{m}\right]$ and $\mathcal{J}=(J_0,\dots,J_l,d_0,\dots,d_{m-1})$ is an upper Rado functional for $P$.

\begin{cor}\label{corollary:PR_URF_system_1} If $\mathcal{J}$ is an upper Rado functional for $P$, then the system $\hat{A}_\mathcal{J}\vec{t}=\vec{b}$ admits a constant solution $s\in\Z$, and the matrix $\hat{A}_\mathcal{J}$ satisfies the columns condition. 
\end{cor}

\begin{proof} By Theorem \ref{theorem:PR_URF_system}, the system $A_\mathcal{J}\vec{t}=\vec{b}$ admits infinitely many monochromatic solutions; hence, by Condition 2 of the Theorem \ref{theorem:PR_linear_inhomogeneous} and the Lemma \ref{lemma:infinitely_pr_inhomogeneous_linear}, there is a constant solution $s\in\Z$ and $A$ must satisfy the columns condition.
\end{proof}

Theorem \ref{theorem:PR_URF_system} has particularly restrictive consequences when $m\geq 1$ in the Rado functionals:

\begin{cor}\label{corollary:PR_URF_system_2} 
If $m\geq 1$, the sets $J_0,\dots,J_l$ are homogeneous. 
\end{cor}

\begin{proof} By Corollary \ref{corollary:PR_URF_system_1}, there is $s\in\Z$ so that $\hat{A}_\mathcal{J}(s,\dots,s)=\vec{b}$. As $m\geq 1$, $\vec{b}\neq 0$, which implies that $s\neq 0$. Consequently, for each $i\in\{0,\dots,i\}$, $M(J_i)(s,\dots,s)=0$. By the definition of $M(J_i)$, for each $j\in\{2,\dots,k_i\}$
\begin{align*}\label{corollary:rado_sets_homogeneous_eq_1}
      0 = & \big(\alpha_{i,j}(1)-\alpha_{i,1}(1)\big)s + \big(\alpha_{i,j}(2)-\alpha_{i,1}(2)\big)s+\dots + \big(\alpha_{i,j}(n)-\alpha_{i,1}(n)\big)s \\
      = & \big(|\alpha_{i,j}|-|\alpha_{i,1}|\big)s.
\end{align*}
Since $s\neq 0$, it must be $|\alpha_{i,j}|-|\alpha_{i,1}|=0$, namely each $J_{i}$ is homogeneous. 
\end{proof}

\begin{cor}\label{corollary:PR_URF_system_3} 
If $m\geq 1$, there exists an $s\in\Z$ such that for all $i\in\{0,\dots,m-1\}$, $\beta_i\in J_i$ and $\alpha\in J_m$
\begin{equation*}
    |\beta_i| = |\alpha| + \frac{d_i}{s}.
\end{equation*}
\end{cor}

\begin{proof} By Corollary \ref{corollary:PR_URF_system_1} , there is a solution $s\in\Z$ to the system $A\vec{t}=b$. By the definition of the matrix $\hat{N}_\mathcal{J}$, for all $i\in\{1,\dots,m-1\}$,
\begin{align*}\label{corollary:rado_sets_homogeneous_eq_2}
      d_i = & \big(\alpha_{i,1}(1)-\alpha_{m,1}(1)\big)s + \big(\alpha_{i,1}(2)-\alpha_{m,1}(2)\big)s+\dots + \big(\alpha_{i,1}(n)-\alpha_{m,1}(n)\big)s \\
      = & \big(|\alpha_{i,1}|-|\alpha_{m,1}|\big)s.\tag{$\star\star$}
\end{align*}
Hence, we have that $s\neq 0$. For all $i\in[m-1]$, by Equation \ref{corollary:rado_sets_homogeneous_eq_2}, we must have $|\alpha_{i,1}|=|\alpha_{m,1}|+\frac{d_i}{s}$, which concludes the proof. \end{proof}

\begin{cor}\label{corollary:PR_URF_system_4}  If $P\in R\left[x_{1},\dots,x_{n}\right]$ is homogeneous and $\left(J_{0},\dots,J_{l},d_{0},\dots,d_{m-1}\right)$ is an upper Rado functional, then necessarily $m=0$.
\end{cor}
\begin{proof}
Suppose that $m\geq 1$ and let $\mathcal{J}=(J_0,\dots, J_l)$ be an upper Rado functional for $P$ of order $m$. Then, we have that the system $A_\mathcal{J}\vec{t}=\vec{b}$ is infinitely partition regular and there is a constant solution $s\in\Z$ to it. As $P$ is homogeneous, for each $i,j\in[0,l]$, $\alpha\in J_i$ and $\beta\in J_j$ we have that $|\alpha|=|\beta|$; hence, by the construction of the matrix $A_\mathcal{J}$ and the existence of the constant solution, it must be the case that $\vec{b}=0$, which is absurd by the Definition \ref{definition:rado_functionals}. 
\end{proof}

The above corollaries show that having an upper Rado functional with $m\geq 1$ forces very restrictive conditions on the Rado partition, both on its Rado sets and the increments in the functional.

To conclude this section, we want to characterize the partition regularity of systems of the form \ref{eq:inhom}. When $\vec{b}$ is null and $\gg$ is substituted by $>$, the partition regularity of such systems has been settled by N. Hindman and I. Leader in \cite{HindmanLeader1998} and it is displayed as Theorem \ref{theorem:pr_inequalities_hindman_leader} below. Before proceeding to such characterizations, we need  a characterization of the partition regularity phenomena via ultrafilters that we employ henceforth. Partition regularity phenomena, as well as several other Ramsey theoretical notions, are connected with ultrafilter through Theorem \ref{theorem:PR_iff_ultrafilters} below. For an introduction of the basic theory of ultrafilters and its applications in Ramsey theory we refer to the monograph \cite{HindmanStrauss2011}. 

\begin{definition}\cite[Definition 3.10]{HindmanStrauss2011}
Let $\mathcal{C}$ a collection of sets; we say that $\mathcal{C}$ is partition regular\footnote{Or also called \emph{weekly partition regular}.} if given any coloring $c$ of $\bigcup\mathcal{C}$, one can find a $c$-monochromatic $A\in \mathcal{C}$; i.e. for all $a,a'\in A$ one has $c(a)=c(a')$.  
\end{definition}

For instance, a system of polynomial equations over $R$ is partition regular over $S\subseteq R$ if and only if the collection of all subsets of $S$ that contain solutions of the system is a partition regular collection; the analogous applies to systems of inequalities, the partition regular relation of Definition \ref{definition:pr_with_delimiters}
or any mixed systems with such binary relations. 
\begin{thm}\cite[Theorem 3.11]{HindmanStrauss2011}\label{theorem:PR_iff_ultrafilters}
A collection $\mathcal{C}$ of subsets of a set $S$ is partition regular if and only if for all $A\in\mathcal{C}$ there is an ultrafilter 
\begin{equation*}
   \U\subseteq \{B\subseteq S:\exists C\in\mathcal{C}(B\subseteq C)\} 
\end{equation*}
on $S$ such that $A\in \U$. Such ultrafilter is said to witnesses (or is a witness for) the partition regularity of $\mathcal{C}$.
\end{thm}

Hence, for instance, a system of Diophantine equations is partition regular over $\N$ if and only if there is an ultrafilter $\U\in\beta\N$ such that any $A\in\U$ contains a solution to the system in question. The analogous applies to systems of inequalities, the partition regular relation of Definition \ref{definition:pr_with_delimiters}
or any mixed systems with such binary relations. 

We now proceed to the characterization of upper Rado functionals via the partition regularity of systems of inequalities and equations. 

\begin{thm}\cite[Theorem 2]{HindmanLeader1998}\label{theorem:pr_inequalities_hindman_leader}
Let $A$ be a $u\times v$ matrix of rational entries and for each $j\in[1,d]$, let $\vec{b_j}=(b_{j1},\dots,b_{jn})$ be a vector of rational entries. Then the following are equivalent:
\begin{enumerate}
    \item the system 
        \begin{equation*}
            \left\{
                \begin{matrix}
                    A\vec{t} & = & 0\\
                    b_1\cdot \vec{t} & > & 0\\
                    \vdots & \vdots &\vdots\\
                    b_d\cdot \vec{t} & > & 0
                \end{matrix}
            \right.
        \end{equation*}
        is partition regular over $\N$;
    \item there are rationals $q_1,\dots,q_d$ such that the system of equations
        \begin{equation*}
            \left\{
                \begin{matrix}
                    A\vec{t} & = & 0\\
                    \vec{b}_1\cdot \vec{t} - q_1z_1 & = & 0\\
                    \vdots & \vdots &\vdots\\
                    \vec{b}_d\cdot \vec{t} - q_dz_d & = & 0
                \end{matrix}
            \right.
        \end{equation*}
    on the variables $\vec{t}=(t_1,\dots,t_n)$, $z_1,\dots,z_d$ is partition regular over $\N$.
    \end{enumerate}
\end{thm}

However, to much of our surprise, and at best of our knowledge, the general partition regularity of systems like \ref{eq:inhom} has not been characterized yet in the literature. Therefore, we provide such a characterization below. Firstly, we show that using $\gg$ and $>$ are equivalent when it comes to partition regularity.

\begin{lem}\label{lemma:hindman-leader-inifinite-inequalities}
Let $A\in\mat{m}{n}{\Z}$, $b_1,\dots,b_k\in\Z^n$. Then, the following are equivalent
\begin{enumerate}
    \item the system 
    \begin{equation*}\label{eq1}
        \left\{
            \begin{matrix}
                A\vec{t} & = & 0\\
                b_1\cdot\vec{t} & \gg & 0\\
                \vdots & \vdots & \vdots\\
                b_k\cdot\vec{t} & \gg & 0\\
            \end{matrix}
        \right.\tag{$\#$}
    \end{equation*}
    is partition regular over $\N$; 
    \item the system 
    \begin{equation*}\label{eq1}
        \left\{
            \begin{matrix}
                A\vec{t} & = & 0\\
                b_1\cdot\vec{t} & > & 0\\
                \vdots & \vdots & \vdots\\
                b_k\cdot\vec{t} & > & 0\\
            \end{matrix}
        \right.\tag{$\#$}
    \end{equation*}  
is partition regular over $\N$. 
\end{enumerate}
\end{lem}

\begin{proof} That (1) entails (2) is immediate. Conversely, suppose that (2) holds. By Theorem \ref{theorem:pr_inequalities_hindman_leader} there are rationals $q_{1},\dots,q_{k}>0$ such that the system 
    \begin{equation*}
        \left\{
            \begin{matrix}
                A\vec{t} & = & 0\\
                b_1\cdot\vec{t} -q_1z_1& = & 0\\
                \vdots & \vdots & \vdots\\
                b_k\cdot\vec{t} -q_kz_k& = & 0\\
            \end{matrix}
        \right.
    \end{equation*}
is partition regular over $\N$ on the variables $\vec{t}$ and $z_1,\dots,z_k$. Since this system is homogeneous, there is a free ultrafilter $\U\in\bN$ such that every set $A\in\U$ contains a solution to this system (see e.g. \cite[Theorem 3]{BeiglbockBergelsonDownarowicz2009}). Let $r\in\N$ and pick a $d\in\N$ such that $d > \max\{rq_1^{-1},\dots,rq_k^{-1}\}$; then the set $I=[d,+\infty[$ is an element of $\U$ and thus there are $t_1,\dots,t_n,z_1,\dots,z_k\in I$ such that $A\vec{t}=0$ and for each $i\in\{1,\dots,k\}$, $b_i\cdot\vec{t}-q_iz_i = 0$. Evidently this implies that $b_i\cdot\vec{t}>r$. 
\end{proof}

We can now characterize the partition regularity of mixed inhomogeneous systems of linear equations and inequalities.

\begin{thm}\label{theorem:infinite-non-homegenous-inequalities}
Let $A\in\mat{m}{n}{\Z}$, $b_1,\dots,b_k\in\Z^n$ and $d\in\Z^m$. Then, the following are equivalent
\begin{enumerate}
    \item the system 
    \begin{equation*}\label{equation:system_inf_pra_ineq_1}
        \left\{
            \begin{matrix}
                A\vec{t} & = & d\\
                b_1\cdot\vec{t} & \gg & 0\\
                \vdots & \vdots & \vdots\\
                b_k\cdot\vec{t} & \gg & 0\\
            \end{matrix}
        \right.\tag{$\star$}
    \end{equation*}
    is partition regular;
    \item either there is a constant $s\in\N$ such that $A(s,\dots,s)^{T}=d$ and, for all $i\in\{1,\dots,k\}$, $\sum_{j=1}^{n} b_{ij} > 0$; or there is a constant $s\in\Z$ such that $A(s,\dots,s)=d$ and the system 
    \begin{equation*}\label{equation:system_inf_pra_ineq_2}
        \left\{
            \begin{matrix}
                A\vec{t} & = & 0\\
                b_1\cdot\vec{t} & \gg & 0\\
                \vdots & \vdots & \vdots\\
                b_k\cdot\vec{t} & \gg & 0\\
            \end{matrix}
        \right.\tag{$\star\star$}
    \end{equation*}
    is partition regular;
\item either there is a $s\in\N$ constant solution to $A\vec{t}=d$ and, for all $i\in[k]$, $\sum_{j=1}^{n}b_{ij}>0$; or there is a constant solution $s\in\Z$ to $A\vec{t}=d$ and there are positive $q_1,\dots, q_k\in\Q$ such that the system 
    \begin{equation*}
        \left\{
            \begin{matrix}
                A\vec{t} & = & 0\\
                b_1\cdot\vec{t} -q_1z_1& = & 0\\
                \vdots & \vdots & \vdots\\
                b_k\cdot\vec{t} -q_kz_k& = & 0\\
            \end{matrix}
        \right.
    \end{equation*}
    is partition regular.
\end{enumerate}
\end{thm}

\begin{proof}
Let us show that (1) implies (2). As the system (\ref{equation:system_inf_pra_ineq_1}) is partition regular, in particular the system $A\vec{t} = d$ is partition regular; consequently, by Theorem \ref{theorem:PR_linear_inhomogeneous}, there are two alternatives: either
\begin{enumerate}[label=(\alph*)]
    \item there is a $s\in\N$ such that $A(s,\dots,s)=d$; or
    \item $A\vec{t}=0$ is partition regular and there is a $s\in \Z$ such that $A\vec{s}=d$. 
\end{enumerate}
If (a) occurs and $A\vec{t}=0$ is not partition regular, by Lemma \ref{lemma:infinitely_pr_inhomogeneous_linear}, any monochromatic solution $t_1,\dots,t_n$ of $A\vec{t}=d$ must satisfy $t_1,\dots,t_n\leq s$; as such, we have that $\sum_{j=1}^{n} b_{ij} s > b_1t_1+\dots+b_nt_n > 0$, which implies that $\sum_{j=1}^{n} b_{ij} >0$.

Now, if (b) occurs, let $c:\N\to[l]$ and $r\in\N$; and define $\chi:\N\to[l+1]$ as 
\begin{equation*}
    \chi(x) = 
    \begin{cases}
        c(x-s),& \text{ if } x>s;\text{ or }\\
        l+1,&\text{ otherwise.}
    \end{cases}
\end{equation*}
As the systems (\ref{equation:system_inf_pra_ineq_1}) and $A\vec{t}=0$ are partition regular, there are $\chi$-monochromatic $t_1,\dots,t_n\in\N$ satisfying $A(t_1,\dots,t_n)=d$, $t_1,\dots,t_n>s$ and, for each $i\in[k]$,
\begin{equation*}
    b_i\cdot(t_1,\dots,t_n) > r+ \left|s\sum_{j=1}^{n}b_{ij}\right|
\end{equation*}
Clearly we have that $t_1-s,\dots,t_n-s$ form a $c$-monochromatic solution to $A\vec{t}=0$. Moreover, 
\begin{equation*}
    b_i\cdot(t_1-s,\dots,t_n-s) = b_i\cdot (t_1,\dots,t_n) - s\sum_{j=1}^{n}b_{ij} > r.
\end{equation*}
Hence, $t_1-s,\dots,t_n-s$ is a $c$-monochromatic solution to the system (\ref{equation:system_inf_pra_ineq_2}).

By Theorem \ref{theorem:pr_inequalities_hindman_leader}, we have that (2) implies (3), and (3) implies (1) is trivial. 
\end{proof}

We can summarize the results of this section as follows: in order to show that a given $(J_0,\dots,J_l,d_0,\dots,d_{m-1})$ is an upper Rado functional for $P\in R[x_1,\dots,x_n]$, it is necessary and sufficient to show that there exists $q_{m+1},\dots,q_l\in\Q_{>0}$ such that the system $O(t_1,\dots,t_n,z_{m+1},\dots,z_l)=\vec{b}$ is infinitely partition regular, where 

\begin{equation*}
    O = 
    \begin{pmatrix}
        A & \boldsymbol{0}_{u\times v} \\
        \begin{pmatrix}
            \alpha_{m,1}-\alpha_{m+1,1} \\
            \vdots\\
            \alpha_{m,1}-\alpha_{l,1}
        \end{pmatrix}
        & Q
    \end{pmatrix},
\end{equation*}
$\boldsymbol{0}_{u\times v}$ is the $u\times v$ ($u=k_0+\dots+k_l-l+m+1$ and $v=l-m-1$) matrix with $0$ in all entries,
\begin{equation*}
    Q = 
    \begin{pmatrix}
        -q_{m+1} & \phantom{-}0 & \dots & \phantom{-}0\\
        \phantom{-}0 & -q_2 & \dots & \phantom{-}0\\
        \vdots & \vdots & \ddots & \vdots\\
        \phantom{-}0 & \phantom{-}0 & \dots & -q_l
    \end{pmatrix}
\end{equation*}
and $b\in\Z^{u+l-m-1}$ is given by 
\begin{equation*}
    b(i) = 
    \begin{cases}
        0, & \text{ if } i\leq k_0+\dots+k_l-l\\
        d_i, & \text{ if } k_0+\dots+k_l-l < i \leq m+1\\
        0, & \text{ if } m+1 < i.
    \end{cases}
\end{equation*}
where $d_m=0$.

Observe also that, when this happens with $m\geq 1$, the restrictions to $\left(J_{0},\dots,J_{l},d_{0},\dots,d_{m-1}\right)$ imposed by Corollaries \ref{corollary:PR_URF_system_1} and \ref{corollary:PR_URF_system_3} apply. We will take that in consideration to study polynomials in three variables in the Section \ref{section:examples_3_var}.

\section{Rado conditions}\label{section:rado_conditions}

In this section we discuss some necessary and sufficient conditions for the partition regularity of equations that are formulated in terms of upper Rado functionals. Let us start by recalling the definition of the maximal Rado condition (see \cite[Definition 2.16]{BarretMoreiraLupiniMoreira2021}).

\begin{definition}
Let $P\in\Z[x_1,\dots,x_n]$, $P(\boldsymbol{x})=\sum_{\alpha}c_\alpha\boldsymbol{x}^\alpha$ and $\mathcal{J}=(J_0,\dots,J_l,d_0,\dots,d_{m-1})$ be a upper Rado functional for $P$. Setting $d_m=0$, for all $q\in\N$ define the monovariate polynomial 
\begin{equation*}
    Q_{\mathcal{J},q}(w) = \sum_{i=0}^{m} q^{d_i}\left(\sum_{\alpha\in J_i} c_{\alpha} w^{|\alpha|}\right).
\end{equation*}
We say that the polynomial $P\in\Z[x_1,\dots,x_n]$ satisfies the \emph{maximal Rado condition} if for all $q\in\N\setminus\{1\}$ there exists an upper Rado polynomial $\mathcal{J}$ for $P$ such that $Q_{\mathcal{J},q}$ has a real root $1\leq w\leq q$.
\end{definition}

Notice that in this case, by Corollary \ref{corollary:PR_URF_system_3}, when $m\geq 1$ there are $L_0,\dots,L_m\in\N$ and $s\in\Z$ such that for all $i\in\{0,\dots,m\}$ and $\alpha\in J_i$, $L_i=|\alpha|$ and thus
\begin{equation*}
    Q_{\mathcal{F},q}(w) =  w^{L_m}\sum_{i=0}^{m} q^{d_i}\left(\sum_{\alpha\in J_i} c_{\alpha} w^{\frac{d_i}{s}}\right).
\end{equation*}

Hence, $P$ satisfies the maximal Rado condition iff for all $q\in\N$ there exist an upper Rado polynomial $\mathcal{F}=(J_0,\dots,J_l,d_0,\dots,d_{m-1})$  for $P$ and $s\in\Z$ dividing $d_0,\dots,d_{m-1}$ such that
\begin{equation*}
    R_{\mathcal{F},q}(w) =  \sum_{i=0}^{m} q^{d_i}\bar{c}_i w^{\frac{d_i}{s}} = \sum_{i=0}^{m} \bar{c}_i\left(qw^{\frac{1}{s}}\right)^{d_i}
\end{equation*}
has a real root in $[1,q]$, where 
\begin{equation*}
    \bar{c}_i = \sum_{\alpha\in J_i} c_{\alpha}.
\end{equation*}
Observe also that, when $m=0$ (for example, when $P$ is homogeneous) the above condition gets the simpler form $\overline{c}_{0}=0$. 

The importance of the maximal Rado condition, as proved in \cite[Theorem 3.1]{BarretMoreiraLupiniMoreira2021}, is that it gives a necessary condition to the partition regularity of polynomial equations over $\N$. However, in general, the maximal Rado condition alone is not sufficient to prove the partition regularity of a given polynomial; actually, it is not even sufficient to prove that it has non-constant solutions.

\begin{exam}
Let $P(x,y,z)=x^3+y^3-z^3$, $J_0=\{(3,0,0),(0,0,3)\}$ and $J_1=\{(0,3,0)\}$. Then, we have that 
\begin{equation*}
    O = 
    \begin{pmatrix}
        \phantom{-}3 & \phantom{-}0 & -3 & \phantom{-}0\\
        \phantom{-}3 & -3 & \phantom{-}0 & -1
    \end{pmatrix}
\end{equation*}
satisfies the columns condition, which implies that the system 
\begin{equation*}
    \left\{
        \begin{matrix}
             3t_1 & = &  3t_2\\
             3t_1 & \gg&  3t_2
        \end{matrix}
    \right.
\end{equation*}
is infinitely partition regular. Hence, $\mathcal{J}=(J_0,J_1)$ is an upper Rado functional for $P$ of order $0$. Moreover, for each $q\in\N\setminus\{1\}$ we have that
\begin{equation*}
    Q_{\mathcal{J},q}(w) = 1-1 = 0,
\end{equation*}
which proves that $P$ satisfies the maximal Rado condition. However, $P(x,y,z)=0$ is not partition regular, since this equation does not admit any non-trivial integral solutions. 
\end{exam}

Therefore, a natural question that arises is: under which additional hypothesis is the maximal Rado condition sufficent to prove the partition regularity of a given equation?

To answer this question, we introduce a strengthened notion.

\begin{definition}
A complete Rado functional is an upper Rado functional of the form $\left(J_{0},\dots,J_{l},d_{0},\dots,d_{l-1}\right)$, i.e. an upper Rado functional of maximal order.
\end{definition}

Fixing a $r\in R$, define $\exp_r:\N\to R$ as $\exp_r(x)=r^x$ and let $\overline{\exp}_r:\beta\N\to\beta R$ be the unique continuous extension of $\exp_r$ over $\beta\N$. Then, for $\U\in\beta\N$, we have that $U\in\overline{\exp}_r(\U)$ if and only if $\{x\in\N:r^x\in U\}\in\U$.  Moreover, if $S\subseteq R$ is closed under exponentiation, we have that $\overline{\exp}_s(\U)\in\beta S$ for all $s\in S$ and $\U\in\beta\N$.  

\begin{thm}\label{theorem:pr_polynomials_crf}
Let $P\in R[x_1,\dots,x_n]$, $P(\boldsymbol{x})=\sum_{\alpha}c_\alpha \boldsymbol{x}^\alpha$, and $\mathcal{J}=(J_0,\dots,J_l,d_0,\dots,d_{l-1})$ be a complete Rado functional for $P$. Define $d_l=0$ and 

\begin{equation*}
    Q_{P,\mathcal{J}}(w):=\sum_{i=0}^{l}\overline{c_{i}}w^{d_{i}},
\end{equation*}
where for each $i\in[0,l]$ let $\overline{c}_i=\sum_{\alpha\in J_i}c_\alpha$. Suppose that $S\subseteq R$ is infinite and closed under exponentiation and that $Q_{P,\mathcal{J}}$ has root in $S$. Then $P$ is partition regular over $S$. 
\end{thm}

\begin{proof}
By Theorems \ref{theorem:PR_URF_system} and \ref{theorem:PR_iff_ultrafilters}, there exists an $\U\in\bN$ that witnesses the infinite partition regularity of the system $A_\mathcal{J}\vec{t}=0$, namely such that for all $U\in\U$, $i\in[0,l]$, $\alpha\in J_i$ and $\beta\in J_{l-1}$ there are $\vec{u}=(u_1,\dots,u_n)\in U^n$ such that $(\alpha-\beta)\cdot \vec{u}=d_i$. Given any root $s\in S$ of $Q_{P,\mathcal{J}}$, we claim that $\V=\overline{\exp}_s(\U)$ is a witness of the partition regularity of $P(x_1,\dots,x_n)=0$. Indeed, if $V\in\V$, we have that $U=\{x\in\N:s^x\in V\}\in\U$ and thus, as observed above, one can find $\vec{u}=(u_1,\dots,u_n)\in U^n$ such that $(\alpha-\beta)\cdot\vec{u}=d_i$. For each $i\in[0,l]$, define $s_i=s^{u_i}$; then, $s_i\in V$ and 
\begin{align*}
    P(s_1,\dots,s_n) & = \sum_{i=0}^{l}\sum_{\alpha\in J_i}c_\alpha s_1^{\alpha(1)}\cdot s_n^{\alpha(n)} = \sum_{i=0}^{l}\sum_{\alpha\in J_i} c_\alpha s^{\alpha\cdot \vec{u}} = \sum_{i=0}^{l}\sum_{\alpha\in J_i} c_\alpha s^{\beta\cdot \vec{u}+d_i} = \\
    & = s^{\beta\cdot\vec{u}}\sum_{i=0}^{l}\left(\sum_{\alpha\in J_i} c_\alpha\right) s^{d_i} = s^{\beta\cdot\vec{u}}\sum_{i=0}^{l}\overline{c}_i s^{d_i} = s^{\beta\cdot\vec{u}}Q_{P,\mathcal{J}}(s) = 0,
\end{align*}
as desired.
\end{proof}

\begin{exam}
Let $\lambda\in\N$, $a,b\in\C$ such that $a\in\N$  and $P(x,y,z)=abxy^2-(a+b)x^2yz\lambda+x^3z^{2\lambda}$. Let $J_2 = \{(1,2,0)\}$, $J_1=\{(2,1,\lambda)\}$, $J_0=\{(3,0,2\lambda)\}$, $d_1=\lambda$ and $d_2=4\lambda$. Then $\mathcal{J}=(J_0,J_1,J_2,d_0,d_1)$ is an upper Rado functional for $P$ of order $2$, since the matrix 
\begin{equation*}
    O = 
    \begin{pmatrix}
    
        \phantom{-}1 & -1 & \phantom{-}\lambda \\
        \phantom{-}2 & -2 & 2\lambda \\
    \end{pmatrix}
\end{equation*}
is infinitely partition regular and the system $O\vec{t}=(2\lambda,4\lambda)$ has a constant solution $s=2$. We have that 
\begin{equation*}
    Q_{\mathcal{J},P}(w) = a - (a+b)w^{2\lambda} + w^{4\lambda}.
\end{equation*}
has $a$ as a natural root. Hence, we have that $P$ is partition regular over $\N$.  
\end{exam}

The next result is an attempt to recover the partition regular of a polynomial equation over an integral domain $D$ from the existence of a root for $Q_{P,\mathcal{J}}$ in the field of fractions of $D$. Ideally, these kind of results could be combined with tests for existences of roots, such as the \emph{rational root test} for unique factorization domains \cite[Proposition 5.5]{Aluffi2021}, to provide the partition regularity of polynomial equations over these domains.

\begin{lem}\label{corollary:rational-root-maximal-rado}
Let $D$ be an integral domain and $K$ be the field of fractions of $D$. Suppose that $\mathcal{J}=(J_0,\dots,J_l,d_0,\dots,d_{l-1})$ is a complete Rado functional for $P\in D[x_1,\dots,x_n]$. In virtue of the Corollaries \ref{corollary:PR_URF_system_1} and \ref{corollary:PR_URF_system_3}, for each $i\in[0,l-1]$ let $L_i=|\alpha|$ for any $\alpha\in J_i$ and let $s\in \N$ be such that  
\begin{equation*}
    L_i = L_0 + \frac{d_i}{s}.
\end{equation*}
Assume that $S\subseteq D$ is closed under multiplication and $a,b\in S$ are such that $b\neq 0$ and $a/b$ is a root of $Q_{P,\mathcal{J}}$. Define
\begin{equation*}
    \tilde{P}(x_1,\dots,x_n) = P\left(\frac{x_1}{b^s},\dots,\frac{x_n}{b^s}\right)
\end{equation*}
over $K$. Then, $\tilde{P}$ is partition regular over $S$. 
\end{lem}
\begin{proof}
By Theorem \ref{theorem:pr_polynomials_crf} it suffices to produce a root for $Q_{\tilde{P},\mathcal{J}}$ in $S$. To this end, let us note that, if $c_\alpha'$ is the coefficient of $\tilde{P}$ associated to a exponent $\alpha\in J_i$, then $c_\alpha'=b^{-sL_i}$. Hence,
\begin{equation*}
    Q_{\tilde{P},\mathcal{J}}(a) = \sum_{i=0}^{l}\bar{c}_i'a^{d_i} = \sum_{i=0}^{l} \bar{c}_i b^{-s L_i}a^d_i. 
\end{equation*}
For each $i\in\{1,\dots,l\}$ we have that $d_i=s(L_i-L_0)$, and thus
\begin{equation*}
    b^{s L_0}Q_{\tilde{P},\mathcal{F}}(a) = \sum_{i=0}^{l} \bar{c}_i b^{-s (L_i - L_0)}a^d_i = \sum_{i=0}^{l} \bar{c}_i\left(\frac{a}{b}\right)^{d_i} = Q_{P,\mathcal{J}}\left(\frac{a}{b}\right)=0,
\end{equation*}
which proves that $a$ is a root for $Q_{\tilde{P},\mathcal{J}}$ in $S$, which concludes the proof.
\end{proof}

\begin{definition}
Let $S$ be a semigroup. An ultrafilter $\U$ over $S$ is said to be a divisible ultrafilter if for all $t\in S$, $tS\in\U$. 
\end{definition}

\begin{thm}\label{theorem:rational-root}
With the same hypotheses and notations of Lemma \ref{corollary:rational-root-maximal-rado}, if $\U$ is a divisible ultrafilter over $S$ such that $\U\models\tilde{P}(x_1,\dots,x_n)=0$, then $P$ is partition regular over $S$.
\end{thm}
\begin{proof}
As $P$ is divisible, for each $A\in\U$ we have that $b^s A = b^s S\cap A\in\U$; as $\U\models\tilde{P}(x_1,\dots,x_n)=0$, there are $a_1,\dots,a_n\in A$ such that 

\begin{equation*}
    0 = \tilde{P}(b^s a_1,\dots,b^s a_n) = P\left(\frac{b^s a_1}{b^s},\dots,\frac{b^s a_n}{b^s}\right),
\end{equation*}
which proves that $\U\models P(x_1,\dots,x_n)=0$.
\end{proof}

\section{Partition regularity of $P(x,y,z)$}\label{section:examples_3_var}
The characterization of the infinite partition regularity of equations in two variables is quite simple (this problem and related generalizations will be discussed in detail in the forthcoming paper \cite{LuperiBagliniArruda2022}): given an algebraically closed field of characteristic $0$, $P\in F[x,y]$ and an infinite $S\subseteq F$, then $P(x,y)$ is infinitely partition regular over $S$ if and only if $x-y$ divides $P$. As a trivial consequence, this applies when $S=\N$. The situation is knowingly much more complicated when we have at least three variables: no general characterization of the partition regularity is known. Our goal in this section is to use the methods developed so far to study the partition regularity of polynomial equations in three variables. 

By Corollary \ref{corollary:PR_URF_system_3}, the only possible upper Rado functionals for an homogeneous $P\in\Z[x_1,\dots,x_n]$ are those of order $m=0$; hence, the maximal Rado condition implies there is a nonempty subset of the coefficients of $P$ that sums zero. For the homogeneous and some other specific cases, \cite[theorem 3.8]{LuperiBagliniDiNasso2018} proves that this set $I$ must be maximal with respect to the degree of the polynomials. This condition is sufficient for the partition regularity in some cases (e.g., $x+y-z$), it is not in others (e.g., $x^{3}+y^{3}-z^{3}$) and, sometimes, even unknown (e.g. $x^{2}+y^{2}-z^{2}$). As such, we believe that the homogeneous case needs different ideas to be treated. In this section we expand this necessary condition for any partition regular inhomogeneous polynomial in three variables and prove a specific restrictions for the degrees of the polynomial $Q$.

\begin{thm}\label{theorem:necessary_pr_inhomogeneous}
If $P\in\Z[x,y,z]$ is inhomogeneous partition regular polynomial with no constant term and suppose that $P$ only admits upper Rado functionals of order $m=0$. Then, fixed an upper Rado functional satisfying the maximal Rado condition, there are polynomials $H,Q_1,\dots,Q_k\in\Z[x,y,z]$ such that 
\begin{enumerate}
    \item $P(x,y,z)=H(x,y,z)+\sum_{i=1}^{k} Q_i$;
    \item $H$ is homogeneous and there is a $I\subseteq\supp(H)$ such that $\sum_{\alpha\in I}c_\alpha=0$; 
    \item $\deg Q_k \leq \dots\leq Q_1 \leq \deg H$; and
    \item there are rational numbers $\rho_1,\dots,\rho_k$ and positive rational numbers $q_1,\dots,q_k$ given by Theorems \ref{theorem:PR_URF_system} and \ref{theorem:infinite-non-homegenous-inequalities}  such that 
    \begin{equation*}
        \deg H_0 = \deg Q_i + \rho_i q_i.
    \end{equation*}
    
\end{enumerate}
\end{thm}

In order to prove Theorem \ref{theorem:necessary_pr_inhomogeneous}, we consider any arbitrary upper Rado functional $\mathcal{J}=(J_0,\dots,J_l,d_0,\dots,d_{m-1})$ that satisfies the maximal Rado condition and invoke Theorems \ref{theorem:PR_URF_system} and \ref{theorem:infinite-non-homegenous-inequalities} to prove that the existence of $q_1,\dots,q_{l-m}\in\Q_{>0}$ such that the system 
\begin{equation*}\label{equation:upper_rado_functional_pr_system_2}
    \left\{
    \begin{matrix}
        A_\mathcal{J}\vec{t} &=& b\\
        (\alpha_{m,1}-\alpha_{m+1,1})\cdot\vec{t} -q_{1}z_1& =& 0\\
        \vdots & \vdots & \vdots\\
        (\alpha_{m,1}-\alpha_{l,1})\cdot\vec{t} - q_{l-m} z_{l-m}& =& 0\\
    \end{matrix}
    \right.
\end{equation*}
is infinitely partition regular, where $A_\mathcal{J}$ was defined in the Equation \ref{equation:def_matrix_urf} and $b$ was defined by the Equation \ref{equation:def_vec_urf}. In particular, this implies that the matrix   
\begin{equation*}
    O = 
    \begin{pmatrix}
        \alpha_{0,2}(1)-\alpha_{0,1}(1) & \alpha_{0,2}(2)-\alpha_{0,1}(2) & \alpha_{0,2}(3)-\alpha_{0,1}(3) & 0 & 0 & \dots & 0\\
        
        \vdots & \vdots & \vdots & \vdots &\vdots & \vdots & \vdots\\
        
        \alpha_{l,2}(1)-\alpha_{l,1}(1) & \alpha_{l,2}(2)-\alpha_{l,1}(2) & \alpha_{l,2}(3)-\alpha_{l,1}(3) & 0 & 0 & \dots & 0\\
        
        \alpha_{0,1}(1)-\alpha_{m,1}(1) & \alpha_{0,1}(2)-\alpha_{m,1}(2) & \alpha_{0,1}(3)-\alpha_{m,1}(3)& 0 & 0 & \dots & 0\\
        
        \vdots & \vdots & \vdots & \vdots & \vdots & \vdots &\vdots\\
        \alpha_{m-1,1}(1)-\alpha_{m,1}(1) & \alpha_{m-1,1}(2)-\alpha_{m,1}(2) & \alpha_{m-1,1}(3)-\alpha_{m,1}(3) & 0 & 0 & \dots & 0\\
        \alpha_{m,1}(1)-\alpha_{m+1,1}(1) & \alpha_{m,1}(2)-\alpha_{m+1,1}(2) & \alpha_{m,1}(3)-\alpha_{m+1,1}(3) & -q_{1}& 0 & \dots & 0\\
        \alpha_{m,1}(1)-\alpha_{m+2,1}(1) & \alpha_{m,1}(2)-\alpha_{m+2,1}(2) & \alpha_{m,1}(3)-\alpha_{m+2,1}(3) & 0 & -q_{2} & \dots & 0\\  
        \vdots & \vdots & \vdots & \vdots & \vdots & \vdots & \vdots\\
        \alpha_{m,1}(1)-\alpha_{l,1}(1) & \alpha_{m,1}(2)-\alpha_{l,1}(2) & \alpha_{m,1}(3)-\alpha_{l,1}(3) & 0 & 0 & \dots & -q_{l-m}\\     
    \end{pmatrix}
\end{equation*}
must satisfy the columns condition and there is a constant solution $s\in\Z$ for $A_\mathcal{J}\vec{t}=b$. Enumerating the columns of $O$ as $C_1,\dots,C_{l-m+3}$, Rado's Theorem witnesses the existence of a partition $I_0,\dots,I_r$ of $[l-m+3]$ such that 
\begin{enumerate}
    \item $\sum_{i\in I_0}C_i = \vec{0}$; and 
    \item for each $u\in\{1,\dots,k\}$, $\sum_{i\in I_u} C_i \in\vecspan_{\Q}\{C_j:j\in I_0\cup\dots\cup I_{u-1}\}$.
\end{enumerate}
By the configuration of the matrix $O$, one of the following possibilities must happen:
\begin{description}
    \item[Case 1.] $I_0\cap\{1,2,3\}=\{1\}$;
    \item[Case 2.] $I_0\cap\{1,2,3\}=\{1,2\}$; or
    \item[Case 3.]  $I_0\cap\{1,2,3\}=\{1,2,3\}$.
\end{description}
Hence, we divide the proof of Theorem \ref{theorem:necessary_pr_inhomogeneous} in the three Lemmas bellow, each of which treats one of the above cases. The following Lemmas also show that, in each particular case, the polynomial must have a fixed structure.

\begin{lem}\label{lemma:urf_3_var_case_1}
If case 1 happens, then $m = 0$ and there are an $a\in \N$, $q_1,\dots,q_{l}\in\Q_{>0}$, $R_0\in\Z[y,z]$, a partition $K_0,K_1$ of $[l]$, an homogeneous $R_0\in \Z[y,z]$ and for each $i\in K_1$ an homogeneous $R_i\in\Z[y,z]$ and $\rho_i,\rho_i'q_i\in\Q_{\geq 0}$ such that 
\begin{enumerate}
\item $m=0$;
\item in each Rado set $J_{i}$, the exponent of $x$ is constant $P$ is homogeneous in $(y,z)$;
\item the polynomial has the structure  
\begin{equation*}
    P(x,y,z)= x^{a}\left(R_{0}(y,z)+\sum_{i\in K_0}x^{\rho_i q_i} R_i(y,z)\right)
\end{equation*}
and  $\deg R_i = \deg R_0 + \rho_i' q_i$; 
\item the maximal Rado condition reduces to $\sum_{\alpha\in J_{0}}c_{\alpha}=0$, where $J_{0}\subseteq\supp \left(x^{a}R_{0}\right)$.
\end{enumerate}
\end{lem}
\begin{proof}
Let $u,v\in[l-m+3]$ be such that $2\in I_u$ and $3\in I_v$, and for each $i\in[m+1,l]$, let $u_i\in[l-m+3]$ be such that $i-m+3 \in I_{u_i}$. Define 
\begin{equation*}
    K_0 = \{i\in[m+1,l]: u_i \leq \min\{u,v\}\}
\end{equation*}
and $K_1 = [0,l]\setminus K_0$. We divide the rest of proof into the following cases: a) $u=v$; or b) $u\neq v$. 
\begin{description}
    \item[Case a.] $u\neq v$. In this case the second and the third columns of $O$ are both linear combinations of the columns numerated by $I_0$. Given a $i\in K_0$ and $j\in[k_i]$ we must have that $\alpha_{i,1}=\alpha_{i,j}$ and $\alpha_{i,1}=\alpha_{m,1}$, which can only happen if $P$ is just a monomial, which contradicts the hypothesis of $P$ being partition regular over $\N$. 
    \item[Case b.] $u=v$. In this case, the sum of the columns indexed by $I_u$ must be a linear combination of the columns indexed by $I_0\cup\dots\cup I_{u-1}$. But, by the format of the matrix $O$, in each line, the only possible non-zero entries are in the columns $C_2$, $C_3$ and the columns containing the $q_i$'s. For each $i\in[0,l]$ and $j\in[k_i]$, we must have that $\alpha_{i,1}(1)=\alpha_{i,j}(1)$ and $\alpha_{i,1}(2)+\alpha_{i,1}(3)=\alpha_{i,j}(2)+\alpha_{i,j}(3)$; that is, inside the Rado set $J_i$ the exponent of the variable $x$ is constant, and the polynomial $Q_i(y,z)=\sum_{\alpha\in J_i}c_\alpha y^{\alpha(2)}z^{\alpha(3)}$ is homogeneous. For each $i\leq m$, we must have that $\alpha_{i,1}(1)=\alpha_{m,1}(1)$ and $\alpha_{i,1}(2)+\alpha_{i,1}(3)=\alpha_{i,j}(2)+\alpha_{i,j}(3)$; that is, through the Rado sets $J_0,\dots,J_m$ the exponent of the variable $x$ does not change and values $a=\alpha_{m,1}(1)$ and $|\alpha|$ is constant through the Rado sets $J_0,\dots,J_m$; observe that, by Corollary \ref{corollary:PR_URF_system_3}, this can only happen if $m=0$. Let $i>m$. 
    \begin{description}
        \item[Case b--1.] $i\in K_0$. In this case, there are $\rho_i,\rho_i'\in \Q_{\geq 0}$ such that 
        \begin{equation*}
            \alpha_{i,1}(1)=\alpha_{m,1}(1)+\rho_i q_i\text{ and }\alpha_{i,1}(2)+\alpha_{i,1}(3)=\alpha_{m,j}(2)+\alpha_{m,j}(3)+\rho_i' q_i. 
        \end{equation*}
        \item[Case b--2.] $i\in K_1$. In this case, we must have that 
        \begin{equation*}
            \alpha_{i,1}(1)=\alpha_{m,1}(1) \text{ and }\alpha_{i,1}(2)+\alpha_{i,1}(3)=\alpha_{i,j}(2)+\alpha_{i,j}(3).
        \end{equation*}
    \end{description}
\end{description}

Hence, consider the polynomial 
\begin{equation*}
    R_0(y,z) = \sum_{i\in K_1}\sum_{\alpha\in J_i} c_\alpha y^{\alpha(2)}z^{\alpha(3)}
\end{equation*}
and for each $i\in K_0$, let 
\begin{equation*}
    R_i(y,z) = \sum_{\alpha\in J_i} c_\alpha y^{\alpha(2)}z^{\alpha(3)}.
\end{equation*}
We must have that $R_0$ is homogeneous and $\supp(R_0)=\bigcup_{i\in K_1} J_i$. The polynomial $P$ must have the following configuration
\begin{equation*}
    P(x,y,z) = x^a\left(R_{0}(y,z) + \sum_{i\in K_0}x^{\rho_i q_i}R_i(y,z)\right).
\end{equation*}
Moreover, if $P(x,y,z)=0$ is partition regular, since $m=0$, the maximal Rado condition implies the existence of a $A\subseteq \supp(x^aR_0)$ such that $\sum_{\alpha\in A}c_\alpha=0$.
 \end{proof}
\begin{lem}
   If case 2 happens, there are the following two possibilities: 
 \begin{enumerate}
    \item in ach Rado set of $\mathcal{J}$, $P$ is homogeneous in $(x,y)$; 
    \item one of the following conditions holds:
    \begin{enumerate}
     \item either $m=0$ and the conditions of Theorem \ref{theorem:necessary_pr_inhomogeneous} holds; or
     \item each $J_i$ is a singleton, say $J_i=\{\alpha_i\}$; putting $c_i=c_{\alpha_i}$, the polynomial has the following form:
     \begin{equation*}
         P(x,y,z) = z^{\alpha_m(3)}x^{\alpha_{m}(1)} y^{\alpha_{m}(2)}\sum_{i=0}^{m}c_i  z^{\frac{d_i}{s}}+\sum_{i=m+1}^{l} c_i x^{\alpha_{i}(1)} y^{\alpha_{i}(2)} z^{\alpha_i(3)}.
     \end{equation*}
    \end{enumerate}

 \end{enumerate} 

\end{lem}
\begin{proof}
We have that $I_0\cap\{1,2,3\}=\{1,2\}$. In this case, for all $i\in\{0,\dots,l\}$ and $j\in\{1,\dots,k_i\}$, we have that 
\begin{equation*}
    \alpha_{i,j}(1)+\alpha_{i,j}(2)=\alpha_{i,1}(1)+\alpha_{i,1}(2)
\end{equation*}
and 
\begin{equation*}\label{equation:case_2_1}
    \alpha_{i,j}(1)+\alpha_{i,j}(2)=\alpha_{m,1}(1)+\alpha_{m,1}(2) + \delta_i \tag{$\ast$}
\end{equation*}
where 
\begin{equation*}
    \delta_i = 
    \begin{cases}
        0, & \text{ if } 0\leq i\leq m\\
        0, & \text{ if } m+1\leq i\leq l \text{ and } i-m+3\not\in I_0\\ 
        q_i, & \text{ if }  m+1\leq i\leq l \text{ and } i-m+3\in I_0\\ 
    \end{cases}
\end{equation*}
Let $u\in[l-m+3]$ be such that $3\in I_u$. Then, there must be $\rho_1,\dots,\rho_{u-1}\in\Q$ such that $\sum_{a\in I_u}C_a = \rho_1 C_1+\dots \rho_{u-1} C_{u-1}$. By the format of matrix $O$ and Equation \ref{equation:case_2_1}, for each $i\in[0,m-1]$, we have that 
    \begin{align*}\label{equation:case_2_2}
         \alpha_{i,1}(3)-\alpha_{m,1}(3) & = \rho_{1}\big(\alpha_{i,1}(1)-\alpha_{m,1}(1)\big)+\rho_{2}\big(\alpha_{i,1}(2)-\alpha_{m,1}(2)\big)=\\  & =  (\rho_1-\rho_2) (\alpha_{i,1}(1)-\alpha_{m,1}(1)).\tag{$\ast\ast$}
    \end{align*}
By Corollary \ref{corollary:PR_URF_system_3} we have that $J_i$ is homogeneous for all $i\in[0,m]$, which implies that $\alpha_{i,j}(3)=\alpha_{i,1}(3)$ for all $j\in[k_i]$; hence, for all $i\in[0,m]$ and $j\in[k_i]$
    \begin{align*}
         0& =\alpha_{i,j}(3)-\alpha_{i,1}(3)  = \rho_{1}\big(\alpha_{i,j}(1)-\alpha_{i,1}(1)\big)+\rho_{2}\big(\alpha_{i,j}(2)-\alpha_{i,1}(2)\big)=\\  & =  (\rho_1-\rho_2) (\alpha_{i,j}(1)-\alpha_{i,1}(1)). 
    \end{align*}
If $\rho_1=\rho_2$, we have that $|\alpha|$ is constant in $J_0,\dots,J_m$, which can only occur if $m=0$; this also implies that the exponent of the variable $z$ is constant inside each Rado set of $\mathcal{J}$; this implies that the Rado sets $R_1,\dots,R_l$ are homogeneous. In this case, let us partition $[0,l]$ into the sets 
\begin{equation*}
    L_3 = \{i\in [l]: l+3\in I_0\}
\end{equation*}
and $K = [0,l]\setminus K_0$. For each $i\in [l]$ let $v_i \in[l+3]$ such that $i+3\in I_{v_i}$. Define 
\begin{equation*}
    L_1 = \{i\in K_1 : u<v_i\}
\end{equation*}
and $L_2 = K\setminus L_1$. Using the same argument as above, for each $i\in L_1$, we have that $|\alpha_{i,1}|=|\alpha_{0,1}|$ and for each $j\in[k_i]$, $|\alpha_{i,1}=\alpha_{i,j}|$. Hence, the polynomial $H(\boldsymbol{x})=\sum_{i\in L_1}\sum_{\alpha\in J_i} c_\alpha \boldsymbol{x}^\alpha$ is homogeneous; since $0\in L_1$ and by the maximal Rado condition we have that there is a $J\subseteq\supp(H)$ such that $\sum_{\alpha\in J}c_\alpha=0$. For each $i\in L_2\cup L_3$, let $R_i(\boldsymbol{x}) = \sum_{\alpha\in J_i}c_\alpha \boldsymbol{x}^\alpha$. Then we have that 
\begin{equation*}
    \deg H_0 = \deg R_i + \lambda_i \deg R_i. 
\end{equation*}

If $i\in L_2$ and $j\in[k_i]$, there is a $\nu_i\in\Q$ such that $\alpha_{i,j}(1)+\alpha_{i,j}(2)=\alpha_{m,1}(1)+\alpha_{m,1}(2)$ and $\alpha_{i,j}(3) = \alpha_{m,1}(3)-\nu_i q_i$. And if $i\in L_3$, then $\alpha_{i,j}(1)+\alpha_{i,j}(2)=\alpha_{m,1}(1)+\alpha_{m,1}(2)-q_i$ and $\alpha_{i,j}(3)=\alpha_{m,1}(3)-\nu_i q_i$. Hence, in both cases, there is a $\lambda_i\in\Q$ such that $|\alpha_{i,j}|=|\alpha_{m,1}|-\lambda_i q_i$. Since the inequality $(\alpha_{m,1}-\alpha_{i,1})\cdot t\gg 0$ is partition regular, we must have that $\lambda_i >0$.

If $m\geq 1$, then we must have that $\rho_1\neq \rho_2$, which implies that $\alpha_{i,j}(1)=\alpha_{i,1}(1)$ and consequently, by Equation \ref{equation:case_2_1}, $\alpha_{i,j}(2)=\alpha_{i,j}(1)$, i.e. $J_i$ is a singleton. Putting $J_i=\{\alpha_i\}$, Corollary \ref{corollary:PR_URF_system_3} implies that 
\begin{equation*}
    \frac{d_i}{s}= \alpha_{i,1}(3)-\alpha_{m,1}(3),
\end{equation*} 
which implies that the polynomial has the following configuration:
    \begin{equation*}
        P(x,y,z) = x^{\alpha_{m}(1)} y^{\alpha_{m}(2)}z^{\alpha_m(3)}\sum_{i=0}^{m}c_i  z^{\frac{d_i}{s}}+\sum_{i=m+1}^{l} c_i x^{\alpha_{i}(1)} y^{\alpha_{i}(2)} z^{\alpha_i(3)}.
    \end{equation*}
\end{proof}

\begin{lem}
If Case 3 happens, then $m=0$ and there are $q_1,\dots,q_{l}\in\Q_{>0}$, homogeneous polynomials $R_0,\dots,R_l\in\Z[x,y,z]$ and a partition $K_0,K_2$ of $[0,l]$ such that $P(x,y,z)=\sum_{i=0}^{l} R_i(x,y,z)$, there is an nonempty subset of the coefficients of $R_0$ that sum zero, and for each $i\in[l]$, $\deg R_i = \deg R_0 - \chi_i q_i$, where 
    \begin{equation*}
        \chi_i = 
        \begin{cases}
            1, & \text{ if } i\in K_1\\
            0, & \text{ if } i\in K_2.
         \end{cases}
    \end{equation*}
\end{lem}
\begin{proof}
We have that $I_0\cap\{1,2,3\}=\{1,2,3\}$. Since $P$ has no constant term, we have that $|\alpha|$ is constant through $J_0,\dots,J_m$; by Corollary \ref{corollary:PR_URF_system_3}, this can only happen if $m=0$. Let $K_1=\{i\in[0,l]:i+3\in I_0\}$ and $K_2=[0,l]\setminus K_1$. Then, by the format of the matrix $O$, we have that for each $i\in[l]$ and $j\in[k_i]$
\begin{equation*}\label{equation:case_3_1}
    \alpha_{i,j}(1)+\alpha_{i,j}(2)+\alpha_{i,j}(3)=\alpha_{0,1}(1)+\alpha_{0,1}(2) + \alpha_{0,1}(3) - \chi_i q_i.\tag{$\ast$}
\end{equation*}
For each $i\in[0,l]$, define $R_i(\boldsymbol{x})=\sum_{\alpha\in J_i}c_\alpha \boldsymbol{x}^\alpha$. Then, each $R_i$ is homogeneous and by Equation \ref{equation:case_3_1}, $\deg R_i = \deg R_0 - \chi_i q_i$.
\end{proof}

\begin{exam} It was shown in \cite{CsikvariGyarmatiSarkozy2012} that the equation $x+y=z^2$ is not partition regular. We employ Corollary \ref{theorem:necessary_pr_inhomogeneous} to show that this equation is not partition regular. Indeed, the possible Rado sets of $P(x,y,z)=x+y-z^2$ are $J_x=\{(1,0,0)\}$, $J_y=\{(0,1,0)\}$, $J_z=\{(0,0,2)\}$ and $J_{xy}=\{(1,0,0),(0,0,1)\}$. Since $\begin{pmatrix}-1 & 1 & 0\end{pmatrix}$ does not admit a constant solution and $\begin{pmatrix}-1 & 0 & 2\end{pmatrix}$ does not satisfy the columns condition, we must have that the only possible upper Rado functionals for $P$ are those of order $m=0$. Since $P$ cannot be decomposed in $H+R$, where $H$ is an homogeneous with a subset of its coefficients summing zero, the given equation is not partition regular.
\end{exam}

\subsection{Complete Rado functionals in three variables}
Having a complete Rado functional is neither necessary or sufficient for the partition regularity as the Schur equation $x+y=z$ is partition regular and do not admit a complete Rado functional and $xy^2=2z$ has a complete functional, namely $(\{(0,0,1)\}, \{(1,2,0)\},2)$ but is not partition regular \cite[Example 3.7]{LuperiBagliniDiNasso2018}.

The objective of this subsection is to classify all polynomials $P$ in three variables over $\Z$ without constant term and having an inhomogeneous set of multi-indexes that admits a complete Rado functional and provide conditions under which such polynomials are partition regular. 

\begin{thm}
A inhomogeneous $P\in\Z[x,y,z]$ admits a complete Rado functional if and only if there are naturals $r$ and $a$, a rational $\rho$ and an homogeneous $H\in\Z[x,y]$ such that $P(x,y,z)=x^{r-a\rho}H(xz^\rho,y)$. In this case, if $H(x,y)=\prod_{i=1}^{k}(a_ix-b_iy)$ is the decomposition of $H$ into linear factors over $\C$, then 
\begin{enumerate}
    \item if $P$ is partition regular, then there is a $i\in[k]$ such that $\frac{b_i}{a_i}$ is an $\rho$-power in $\Q$; and 
    \item if there is an $i\in[k]$ such that $\frac{b_i}{a_i}$ is an $\rho$-power in $\N$, then $P$ is partition regular. 
\end{enumerate}
\end{thm}

We prove the Theorem above as a consequence of the following results.

\begin{lem}\label{lemma:characterization_complete_rado_func_3_var}
An inhomogeneous $P\in R[x,y,z]$ admits a complete Rado functional if and only if (possibly after a permutation of its variables) there is an enumeration $\alpha_0,\dots,\alpha_l$ of $\supp (P)$ such that 
\begin{enumerate}
    \item for each $i,j\in\{0,\dots,l\}$, $\alpha_i(1)+\alpha_i(2)=\alpha_j(1)+\alpha_j(2)$;
    \item the sign of $\alpha_i(1)-\alpha_l(1)$ is constant; and 
    \item there exists $\rho\in\Q^{\times}$ such that, for each $i\in[0,l-1]$, $\alpha_i(3)=\alpha_l(3)+\rho[\alpha_i(1)-\alpha_l(1)]$ and $\rho[\alpha_i(1)-\alpha_l(1)]\neq 0$. 
\end{enumerate}
\end{lem}

\begin{proof}
Let $\mathcal{J}=(J_0,\dots,J_l,d_0,\dots,d_{l-1})$ be a complete Rado functional for $P$. For each $i\in[0,l]$, let $J_i=\{\alpha_{1,i},\dots,\alpha_{k_i,i}\}$; by Theorem \ref{theorem:PR_URF_system} we have that $A_\mathcal{J}$ satisfies the columns condition and the system $A_\mathcal{J}\vec{t}=\vec{b}$ has a constant solution to $s\in\Z$; let us also observe that $A_\mathcal{J}$ has three columns, namely $C_1$, $C_2$ and $C_3$. Since $P$ is not homogeneous, there are two columns $C_i$ and $C_j$ of $A_\mathcal{J}$ that sum zero and the third one, $C_k$, is $\Q$-linearly dependent of $C_i$ and $C_j$; without loss of generality, we assume that $C_1+C_2=0$ and there is $\rho\in\Q$ ($\rho\neq 0$ since $P$ is not homogeneous) such that $C_3=\rho C_1$. It is easy to see that the fact that $C_1+C_2=0$ implies that $\alpha_{i,j}(1)+\alpha_{i,j}(2)$ is constant along the monomials of $P$; i.e., since $P$ has no constant term, $P(x,y,1)$ is an homogeneous polynomial of $R[x,y]$. Since $J_i$ is homogeneous, we have that $\alpha_{j,i}(3)$ is constant inside $J_i$. We claim that $J_i$ is a singleton; indeed, let $s\in\Z$ be a constant solution to the system $A_\mathcal{J}\vec{t}=\vec{b}$. Then, for each $j\in[0,l-1]$ we have that 
\begin{align*}\label{equation:characterization_complete_rado_func_3_var_1}
    d_i & = s[\alpha_{j,1}(1)-\alpha_{l,1}(1)] + s[\alpha_{j,1}(2)-\alpha_{l,1}(2)] + s[\alpha_{j,1}(3)-\alpha_{l,1}(3)]=\\& = s\rho [\alpha_{j,1}(1) - \alpha_{l,1}(1)], \tag{$\dagger$}
\end{align*}
thus $\alpha_{j,i}(1)$ is constant inside $J_i$; since $C_1+C_2=0$, this also implies that $\alpha_{j,i}(2)$ is constant inside $J_i$, i.e. $J_i$ is a singleton, say $J_i=\{\alpha_i\}$. Finally, by Corollary \ref{corollary:PR_URF_system_3}, we have that $|\alpha_i|=|\alpha_l|+\frac{d_i}{s}$ and thus by Equation \ref{equation:characterization_complete_rado_func_3_var_1}, 
\begin{equation*}
    \alpha_i(3) = \alpha_l(3) + \rho[\alpha_i(1)-\alpha_l(1)].
\end{equation*}
Since $d_i\in\N$, we must have $\rho[\alpha_i(1)-\alpha_l(1)]\neq 0$. The proof of the converse is direct and similar.

\end{proof}

Hence, if $\supp(P)=\{\alpha_0,\dots,\alpha_l\}$ is an enumeration for $\supp(P)$ that satisfies conditions (1) and (2) above, there exist $c_0,\dots,c_l\in\Z^{\times}$ such that 
\begin{align*}
    P(x,y,z) & = c_lx^{\alpha_l(1)}y^{\alpha_l(1)}z^{\alpha_l(3)}+\sum_{i=0}^{l-1}c_ix^{\alpha_i(1)}y^{\alpha_i(1)}z^{\alpha_i(3)} =  \\
    & =  c_lx^{\alpha_l(1)}y^{\alpha_l(1)}z^{\alpha_l(3)}+\sum_{i=0}^{l-1}c_ix^{\alpha_i(1)}y^{\alpha_i(1)}z^{\alpha_l(3)+\rho[\alpha_i(1)-\alpha_l(1)]} = \\
    & = c_lx^{\alpha_l(1)}y^{\alpha_l(1)}z^{\alpha_l(3)}+z^{\alpha_l(3)-\rho\alpha_l(1)}\sum_{i=0}^{l-1}c_ix^{\alpha_i(1)}y^{\alpha_i(1)}z^{\rho\alpha_i(1)} = \\ 
    & = z^{\alpha_l(3)-\rho\alpha_l(1)}\left(c_lx^{\alpha_l(1)}y^{\alpha_l(2)}z^{\rho\alpha_l(1)}+\sum_{i=0}^{l-1}c_ix^{\alpha_i(1)}y^{\alpha_i(1)}z^{\rho\alpha_i(1)}\right)= \\
    & = z^{\alpha_0(3)-\rho\alpha_0(1)}\sum_{i=0}^{l}c_i(xz^\rho)^{\alpha_i(1)}y^{\alpha_i(1)}.
\end{align*}

Hence, we have proven the following:

\begin{cor}\label{corollary:characterization_complete_rado_func_3_var}
An inhomogeneous $P\in R[x,y,z]$ admits a complete Rado functional if and only if (possibly after a permutation of its variables) there are an homogeneous polynomial $H(s,t)=\sum_{i=0}^{l}c_i s^{a_i}t^{b_i}$, $r\in\N$ and $\rho\in\Q^{\times}$ such that $P(x,y,z)=z^{r-\rho a_0}H(xz^{\rho},y)$. 
\end{cor}

We now proceed now to give a condition under which such equations are partition regular. As a polynomial equation is partition regular if and only if one of the factors of the polynomial is (see e.g. \cite[Theorem 3.7]{LuperiBaglini2015}), and since $z^{r-\rho a_0}$ is never $0$ on $\N$, we are left with the problem of characterizing which equations of the form $H(xz^{\rho},y)=0$ are partition regular, where $H\in\Z[x,y]$ is homogeneous and $\rho\in\Q^\times$. Over $\C$, the homogeneity of $H$ implies that $H$ can be factorized as a product of linear factors, namely $H(x,y)=\prod_{i=1}^{k} \left(a_{i}x-b_{i}y\right)$; hence, the equation $H(xz^{\rho},y)=0$ is partition regular if and only if there is an $i\in[k]$ such that the equation $b_i y = a_ixz^{\rho}$ is partition regular over $\N$. Thus, it is enough to characterize the partition regularity of equations of the form 

\begin{equation}\label{thisone} 
a^{n}x^{n}z^{m}=b^{n}y^{n}
\end{equation}
for $m,n,a,b\in\Z$ satisfying $\gcd(m,n)=\gcd(a,b)=1$.

\begin{thm}\label{theorem:char_pr_H}
If the Equation \ref{thisone} is partition regular over $\N$, then  $\frac{a}{b}$ is an $\frac{m}{n}$-power in $\Q$. If $\frac{a}{b}$ is an $\frac{m}{n}$-power in $\N$, then the equation $\ref{thisone}$ is partition regular over $\N$. 
\end{thm}

Let us consider for each prime $p$ the $p$-adic evaluation $\nu_p:\Q^\times\to\Z$. Then, for each non-zero rationals $\rho,r$ we have that $\nu_p(r^\rho)=\rho\nu_p(r)$. Then, an irreducible fraction $\frac{a}{b}$ an $\frac{m}{n}$-power, where $\frac{m}{n}$ is also irreducible, if and only if for all prime $p$, $m$ divides $n\nu_{p}(r)$.

\begin{proof}[Proof of Theorem \ref{theorem:char_pr_H}]
We proceed analogously to the proof of \cite[Lemma 3.3]{FarhangiMagner2021}. Suppose that $r=\frac{a}{b}$ is not an $\frac{m}{n}$-power in $\Q$. Then, there is a prime $p$ such that $m$ does not divide $n\nu_p(r)$. For each $i\in[0,m-1]$, define 
\begin{equation*}
    C_i = \{k\in\N: n\nu_p(k) \equiv i \mod m\}.
\end{equation*}
If $x,y,z\in C_i$, then 
\begin{align*}
    n\nu_p(axz^{\frac{m}{n}})-n\nu_p(by) & \equiv n\nu_p\left(r\right) + n\nu_p(x)+m\nu_p(z)-n\nu_p(y) \equiv\\ & \equiv n\nu_p\left(r\right)\not\equiv 0\mod n
\end{align*}
which implies that the equation \ref{thisone} cannot be solved inside $C_i$. 

Conversely, suppose that $a^nl^m=b^n$ for some $l\in\N$; hence, the partition regularity of the equation \ref{thisone} is equivalent to the partition regularity of the equation
\begin{equation}\label{thisother}
    x^nz^m = l^m y^n.
\end{equation}
By Theorems \ref{theorem:PR_linear_inhomogeneous} and \ref{theorem:PR_iff_ultrafilters}, there is a free ultrafilter $\V\in\bN$ that witnesses the partition regularity of the equation 
\begin{equation*}
    nu+mv=nw+m.
\end{equation*}
We claim that $\U=\exp_l(\V)$ witnesses the partition regularity of the equation \ref{thisother}; indeed, given a $A\in\U$ we have that $B=\{c\in \N: c^l\in A\}\in\V$; thus, one can find $c_1,c_2,c_3\in B$ such that $nc_1+mc_2=nc_3+m$. Thus,  
\begin{equation*}
    (l^{c_1})^n (l^{c_2})^m = l^{m}(l^{c_3})^n,
\end{equation*}
which implies that $x=l^{c_1}$, $z=l^{c_2}$ and $y=l^{c_3}$ is a solution to the Equation \ref{thisother} contained in $A$. By Theorem \ref{theorem:PR_iff_ultrafilters}, the equation \ref{thisother} is infinitely partition regular over $\N$. 
\end{proof}

Note that, if the partition regularity is considered over $\Q$, then the partition regularity of \ref{thisone} occurs if and only if $\frac{a}{b}$ is an $\frac{m}{n}$-power in $\Q$.
\printbibliography 
\end{document}